\documentclass{amsart}

\usepackage{amsmath}
\usepackage{amsfonts}
\usepackage{amsthm}
\usepackage{amssymb}
\usepackage{latexsym}
\usepackage{upref}
\usepackage{mathrsfs}
\usepackage{eufrak}
\usepackage{enumerate}

\newtheorem{theorem}{Theorem}[section]
\newtheorem{lemma}[theorem]{Lemma}
\newtheorem{axiom}{Axiom}[section]

\theoremstyle{definition}
\newtheorem{definition}[theorem]{Definition}

\theoremstyle{remark}
\newtheorem{remark}[theorem]{Remark}

\numberwithin{equation}{section}

\begin{document}
\title{Generalized measurement on size of set}
\author{{Hua-Rong Peng, Da-Hai Li}}
\address{School of Electronics and Information Engineering, Sichuan University, Chengdu 610065, China.}
\author{{Qiong-Hua Wang}}
\address{School of Electronics and Information Engineering, Sichuan University, Chengdu 610065, China. Key Laboratory of Fundamental Synthetic Vision Graphics and Image for National Defense, Sichuan University, Chengdu 610065, China.}
\thanks{Email: B100425@stu.scu.edu.cn}
\keywords{cardinality, size, dimension, measure, one-to-one correspondence, continuum hypothesis, infinity.}

\pagestyle{headings}
\markboth{{\rm H. R. PENG \etal}}{{\rm Generalized measurement on size of set}}
\begin{abstract}
{We generalize the measurement using an expanded concept of cover, in order to provide a new approach to size of set other than cardinality. The generalized measurement has application backgrounds such as a generalized problem in dimension reduction, and has reasons from the existence of the minimum of both the positive size and the positive graduation, i.e., both the minimum is the size of the set $\{0\}$. The minimum of positive graduation in actual measurement provides the possibility that an object cannot be partitioned arbitrarily, e.g., an interval $[0, 1]$ cannot be partitioned by arbitrarily infinite times to keep compatible with the minimum of positive size. For the measurement on size of set, it can be assumed that this minimum is the size of $\{0\}$, in symbols $|\{0\}|$ or graduation $1$. For a set $S$, we generalize any graduation as the size of a set $C_i$ where $\exists x \in S (x \in C_i)$, and $|S|$ is represented by a pair, in symbols $(\emph{C}, N(\emph{C}))$, where $\emph{C} = \cup \{C_i\}$ and $N(\emph{C})$ is a set function on $C_i$, with $C_i$ independent of the order $i$ and $N(\emph{C})$ reflecting the quantity of $C_i$. This pair is a generalized form of box-counting dimension. The yielded size satisfies the properties of outer measure in general cases, and satisfies the properties of measure in the case of graduation $1$; while in the reverse view, measure is a size using the graduation of size of an interval. As for cardinality, the yielded size is a one-to-one correspondence where only addition is allowable, a weak form of cardinality, and rewrites Continuum Hypothesis using dimension as $\omega \dot |\{0,1\}| = 1$. In the reverse view, cardinality of a set is a size in the graduation of the set. The generalized measurement provides a unified approach to dimension, measure, cardinality and hence infinity.}
\end{abstract}
\maketitle

\vspace*{6pt}
\section{Introduction}
\label{intro}
\setcounter{equation}{0}

In pure mathematics the size of set is reflected by cardinality, while in applications the size of an object is known by measurement. In general cases, the result of measurement is dependent on the scale at which the measurement is conducted, e.g., let the object be the balls to define Hausdorff dimension and the size depends on the radius of the balls. Let the object be the set, then the dependence provides the possibility that the size of set depends on the measurement, e.g., the size of $\{1\}$ is $1$ if one measures the set using $\{1\}$, and the size approaches $0$ if using $[0, 1]$. This association between measurement and the size of set is reasonable since that the size of set is now not only a pure problem but also an applicable problem in such applications as dimension reduction (see \cite{Zhang2010}) as follows.

Let finite $N$ points distribute in a space $\mathfrak{S}^D$, and dimension reduction maps these points to another space $\mathfrak{S}^d$, where $d~\le~D$ and dimension reduction preserves the quantity of the points. Let's denote the points in $\mathfrak{S}^D$ as a set $A^D$ and the points in $\mathfrak{S}^d$ as another set $A^d$, then dimension reduction requires the identity between the two sizes of sets, in symbols $|A^D| = |A^d| = N$. The identity is given in cardinal theory if $A^D$ can be put into one-to-one correspondence with $A^d$. Measure theory except counting measure regards the set of finite point as zero measure set; the quantity of finite points is non-sense. Counting measure doesn't distinguish the identity if $N\to \infty$.

According to Euclid's definition 1 that point is of no length, points are with radius $r_i \to 0$. Let's consider a general case where the points are generalized as balls with radius $r_i \ge 0$ and where dimension reduction preserves the quantity of balls, i.e., $N$ balls in $\mathfrak{S}^D$ is mapped to $N$ balls in $\mathfrak{S}^d$. Cardinal theory gives more than the identity because it gives identity between $N$ balls and arbitrary positive number of balls. Measure theory except counting measure gives the quantity $N$; counting measure gives $N$ if $N < \infty$.

The difference between cardinal theory and measure theory suggests the gap between pure mathematics and applied mathematics. In the former one can assume the existence of continuum and partition the continuum in any arbitrary small scale, while in the latter one cannot apply so many properties of continuum since that the partition is limited by either the maximum or the minimum of scale. For an example, let half of the circle projected to line segment $[0, 1]$, and then the minimum of distance between two arbitrary points in the segment can be arbitrary small if the segment is a continuum, while the minimum which makes sense in applications is the minimum of graduation. Of course in applications one can decrease the minimum; the minimum still exists. The minimum of graduation in actual measurement provides the possibility that an object cannot be partitioned arbitrarily. And dimension reduction provides the application background to bridge the gap if it deals with balls including points. But how to bridge it? In this paper we assume that these balls can be described by reals, and constrain the discussion on the set of positive reals.

We consider the size of set from the view of measurement. First, we take the set $S$ as the measured object. Second, we use other sets $C_i$ to cover $S$, where the used sets are not limited to be a subset in $P(S)$, with the meaning of cover in mathematics expanded. Third, we represent the size of $S$ using both $C_i$ and the quantity of $C_i$. Given that the minimum of the size exists, the measurement has constraints. The main results are as follows:
\begin{enumerate}
\item The yielded size is a generalized form of box-counting dimension;
\item The yielded size satisfies the properties of outer measure in general cases, and satisfies the properties of measure if graduation is graduation $1$;
\item The yielded size is a one-to-one correspondence where only addition is allowable, a weak form of cardinality;
\item Continuum Hypothesis is $\omega \dot |\{0,1\}| = 1$, where $\omega$ is the cardinality of the set of all natural numbers, $|\{0,1\}|$ in the left side is the size of a set which has only two elements, and $1$ in the righthand is dimension $1$.
\end{enumerate}

Note that we constrain the discussion on positive reals and that $\mathbb{R}$ means $\mathbb{R^+} \cup \{0\}$. The rest of the paper is organized as follows. In Section~\ref{sos} we unfold the generalized measurement, size of set and what the size requires such as both axioms and definitions. In Section~\ref{rwm} we explore the relation between the size and measure theory. In Section~\ref{rwc} we explore the relation between the size and cardinality. In Section~\ref{conclu} we present our conclusions.

\vspace*{5pt}
\section{Generalized Measurement and Size of Set}
\label{sos}

\setcounter{equation}{0}
\subsection{Motivations of size}
\label{mos}

Expand the concept of cover as follows:
\begin{definition}
Set $\emph{C}$ is a cover on set $A$, if,
\begin{enumerate}
\item $\varnothing \in \emph{C}$;
\item $A \subseteq \cup C_i$, where $i \in I$;
\item $\forall C_i \in \emph{C}$;
\item $\forall C_i (A \cap C_i \ne \varnothing)$.
\end{enumerate}
\end{definition}
Hereafter we use cover in this sense. Given two measurable sets, in symbols $A$ and $B$, recall that the measure of $A$ becomes $\mu_A / \mu_B$ if one covers the set $A$ using set $B$ by $\mu_A / \mu_B$ times and takes $\mu_B$ as unit; this is not a circular definition both if $\mu_B$ is defined and if $\mu_B \ne 0$. With infinite times allowable, we cover the set of points with balls of radius $r \to 0$ by $N(r)$ times where $r \in \mathbb{R}$, in symbols $(r, N(r))$ with $r \to 0$, abbreviated as $(r \to 0, N(r))$, and cover the set of balls with balls of radius $r \ge 0$ with $N(r)$ times, in symbols $(r \ge 0, N(r))$. This yields such instances of pair $(r, N(r))$ as $(r \to 0, 0)$, $(r \to 0, 1)$, $(r \to 0, 2)$, $(r \ge 0, \pi r^2)$, $(r \ge 0, r^D)$, etc. This is motivated by both the concept of fractal dimension (see \cite{Mandelbrot1967}) and the fact that minimum of graduation exists in applications.

Consider the set $\varnothing$ and the set $\{\varnothing\}$. We have $|\{\varnothing\}| > |\varnothing|$ if $r<\infty$, i.e.:
\begin{equation}\label{r10}
(r, 1) > (r, 0),
\end{equation}
where the pair $(r, 0)$ corresponds to size of $|\varnothing|$ and the pair $(r, 1)$ corresponds to size of $|\{\varnothing\}|$. Though dimension of $\varnothing$ is defined as $-1$ and dimension of $\{\varnothing\}$ is defined as $0$ in topology, we have to define the two dimensions as follows according to equation \ref{r10} since that the measurement with $r < \infty$ gives the size of $\varnothing$ as $0$ and gives the size of $\{\varnothing\}$ as $1$, to both keep compatible with measure theory and add convenience to the comparison between pairs.
\begin{equation}\label{00}
(r, 0) = 0.
\end{equation}
\begin{equation}\label{1ge0}
(r, 1) > 0.
\end{equation}
The pairs satisfy that, if using the same $r$ and supposing $(N_0 - N_1)$ is defined:
\begin{equation}
(r, N_0) - (r, N_1) = (r, N_0 - N_1).\\
\end{equation}
Then applying equation \ref{00} gives that:
\begin{equation}
(r, N_0) > (r, N_1), \mbox{if}\: N_0 > N_1.
\end{equation}
If $(r, 0)$ were defined as $-1$ and $(r, 1)$ were defined as $0$, the pairs cannot distinguish sizes of finite sets. Such a need for finite sets excludes the requirement on dimension that countable unions of points have dimension $0$.

Consider the case where $r \to \infty$. Then $\{0, 1\}$ cannot be distinguished from the set of a single point such as $\{0\}$, with both pairs as $(r, 1)$, though $\{0, 1\}$ is not identical to $\{0\}$. This gives reason for a line drawn between indistinguishability and identity.

Consider the similarity between the pair $(r, N(r))$ and box-counting dimension (see \cite{Theiler1990}) which assigns $0$ to $(r, 1)$. Write $(r, N(r))$ except $(r \to 0, 0)$ and except $(r \to 0, 1)$ as follows, if $r < 1$,
\begin{definition}\label{defofD}
(r, N(r)) = log(N(r))/log(1/r).
\end{definition}
Then $(r, N(r)) \not\subset \mathbb{R}$ since that the righthand doesn't always convergent. The pair still makes sense since that reals cannot meet the need for finite set. If a measurable set corresponds to pair $(r, N(r))$, box-counting dimension can be obtained by $r \to 0$.

If box-counting dimension of any strictly fractal set coincides with Hausdorff dimension, equation \ref{defofD} contains the value domain of Hausdorff dimension; moreover, one has $\mathbb{R} \subset (r, N(r))$ since that Hausdorff dimension gives reals by the Hausdorff dimension theorem (see \cite{Soltanifar2006}). Denote box-counting dimension and Hausdorff dimension as $D_{BC}$ and $D_H$, respectively, and then one has that,
\begin{lemma}\label{lemdhdbc}
$\forall D_H \ge 0$, $D_H \in \mathbb{R}$, $\exists D_{BC} \in \mathbb{R} (D_{BC} = D_H)$.
\end{lemma}
\begin{proof}
\begin{enumerate}
\item For $D_H = 0$, the set $\{0\}$ has both $D_H$ and $D_{BC}$ as $0$.
\item By the Hausdorff dimension Theorem (see \cite{Soltanifar2006}), for any real number $x > 0$, there exists a fractal set with Hausdorff dimension $x$ in $n$-dimensional Euclid's space where $n$ satisfies that $-[-x] \le n$. Box-counting dimension of the fractal set is also $x$, using the self-similarity property of fractal set.
\end{enumerate}
\end{proof}

Note that in this subsection, radius $r$ is expressed in the terms of measure theory, e.g., $r \to 0$ means measure of a set corresponding to $r$ is zero in the sense of limit. By associating $r$ with cardinality, one knows that $\{0\}$ implies $r \to 0$ since that the sets with finite cardinality are zero measure set.

\subsection{Concept of pair}
\label{cop}

There exists a class denoted by $\mathfrak{M}$ which satisfies that:
\begin{axiom}\label{defofM}
$\forall x \in \mathfrak{M} ( x\:\mbox{mod}\:1 = 0 )$.
\end{axiom}
Here $1$ is the size of the set of single point.The set of natural numbers ensures that $\mathfrak{M} \ne \varnothing$.

If the elements in any set are counted, the minimum of graduation which makes sense is the size of the set of single point, denoting the minimum of graduation as graduation $|\{0\}|$ or graduation $1$. In general cases, the graduation is an element in $\mathfrak{M}$. We make this as an assumption:
\begin{axiom}\label{assofgradua}
Graduation in the measurement on size of set is an element in $\mathfrak{M}$.
\end{axiom}
For convenience for the measurement on non-empty set, we take graduation $1$ as the minimum graduation, to make that $\forall x \in \mathfrak{M} (x > 0)$. This convention can yield such negative elements in $\mathfrak{M}$ as $|A-B|$ where $A \subset B$, if in this case $(A-B)$ is defined as the inverse element of $(B-A)$, satisfying $|A-B|+|B-A| = 0$. To keep compatible with measure theory, we still keep $A-B$ defined as $\varnothing$ if $A \subset B$. Of course one can define $0$ or even negative number as the minimum.

As for the result of measurement on size of set, no part below the graduation is possible if the graduation is the minimum, i.e., the result is still an element in $\mathfrak{M}$. For other graduations, choose the ceil of the result, to keep the result in $\mathfrak{M}$. Denoting the result as $N$ yields that:
\begin{axiom}
$N \in \mathfrak{M}$.
\end{axiom}
Then all results are elements in $\mathfrak{M}$.

Any set can be measured using the graduation $1$ if it can be measured by graduation $g \ge 1$, where $g \in \mathfrak{M}$, due to a property of operation $mod$. But is there a set which cannot be measured? Cantor has assumed that any set has a cardinality. For a set, we assume that:
\begin{axiom}
Any set can be measured by graduation $1$.
\end{axiom}

Denote the result by $\mathfrak{Cs}$ to differ the result from both cardinal number denoted by $Ca$ and cardinality with Axiom of choice denoted by $PO$. According to the definition of $Ca$ and that of $PO$, where the former bases itself on bijection and the latter bases itself on the method of one-to-one correspondence, any element in either $Ca$ or $PO$ is a result of measurement which obeys Axiom \ref{assofgradua}. Recall that both bijection and one-to-one correspondence can be written as $a \leftrightarrow b$, i.e. $(r, 1) \leftrightarrow (r, 1)$ since that $\{a\}$ is identical to $\{b\}$, and it can be shown easily that:
\begin{lemma}
$(Ca \subset  Cs) \cup (PO \subset Cs)$.
\end{lemma}

%

$\mathfrak{M}$ satisfies that, using the property of operation $mod$:\\
\indent
$\forall x \in \mathfrak{M}\: \forall y \in \mathfrak{M}$,
\begin{enumerate}
\item $x+y \in \mathfrak{M}$, \mbox{with union allowable};
\item $x-y \in \mathfrak{M}$, \mbox{with intersection allowable};
\item $2^x \in \mathfrak{M}$, \mbox{with power set allowable};
\item $x\dot y \in \mathfrak{M}$, \mbox{with pairing allowable}.
\end{enumerate}

Let the pair denoted by $\mathfrak{R}$ which is defined as follows.
\begin{definition}
Given an arbitrary pair $(r, N(r))$ with $r \in \mathfrak{M}$ and $N(r) \in \mathfrak{M}$, it is a pair in $\mathfrak{R}$, if $r > 0$ and if $N(r): \mathfrak{M} \to \mathfrak{M}$ is a function on $r$.
\end{definition}
Radius $r$ has a definite physical meaning that the smaller $r$ can be, the smaller the minimum of graduation is. Counting of $r$, in symbols $N(r)$, is not counting measure as $N(r)$ can refer to infinity, e.g., $N(r)$ is $\omega$ both if $r=1/2$ and if $\omega$ is counted. Now consider the common place among $N(r)$. They satisfy that $N(r)\: mod\: 1 = 0$.

Equation \ref{defofD} includes reals. Let the given reals denoted by $\mathcal{R} \subset \mathbb{R}$. This yields that:
\begin{equation}
(\mathcal{R} = \mathbb{R}) \cup (\mathcal{R} \subset \mathfrak{R}).
\end{equation}
\begin{proof}
\begin{enumerate}
\item $\mathcal{R} = \mathbb{R}$. For real number $x = 0$, there exists a pair $(r, 0) = 0$. According to definition \ref{defofD} and by Lemma \ref{lemdhdbc}, for any real number $x >0 $, there exists a set of which the pair is $(r, r^x) = x \in \mathcal{R}$, i.e., $\mathbb{R} \subset \mathcal{R}$. We have that $\mathcal{R} = \mathbb{R}$.
\item $\mathcal{R} \subseteq \mathfrak{R}$. $\mathcal{R}$ reflects the size of set which is measurable, then we have that $\forall x \in \mathcal{R} (x \in \mathfrak{R})$. According to principle of extensionality, $\mathcal{R} \subseteq \mathfrak{R}$.
\item $\mathcal{R} \ne \mathfrak{R}$. Consider the case of sets where $|A_i| = i$, $i$ \mbox{is finite}, and $B_1 = {x | x \in (0, 1]}$. These sets are measurable and bounded; they can be dealt with equation \ref{defofD}; assume that $\mathcal{R} = \mathfrak{R}$, then the pairs all correspond to reals. We know that the pair of $B_1$ corresponds to $1$, and suppose that $(r, 2) = c$ with $r \to 0$, where $0<c<1$ since that $|A_2| > |\varnothing|$. Let $r < 1/n$. According to the Archimedean property of real field, there exists such a set $A_n$ that it satisfies $|A_n| = n = \lceil1/c\rceil\cdot c>1$. But the inequity $n<1/r$ gives that $|A_n| < 1$. There yields a contradiction. Hence the assumption that $\mathcal{R} = \mathfrak{R}$ contains inconsistency, and we have that $\mathcal{R} \ne \mathfrak{R}$.
\end{enumerate}
\end{proof}

Then we can discuss the order of $\mathfrak{R}$ as follows,\\
\indent
$\forall x, y, z\in \mathfrak{R}$,
\begin{enumerate}
\item $x > y$\: \mbox{and} $y > z$\: \mbox{imply that} $x > z$;
\item $\forall a \in \mathcal{R}\forall b \in\mathcal{R}(a \ne b\to \exists x(a < x < b))$.
\end{enumerate}

The property that any pair can be compared with another, i.e. $\forall x \forall y( x > y \cup x = y \cup x < y)$, needs Axiom of Choice (AC). To see this, recall that AC presents one of its form as that every cardinality defined by one-to-one correspondence is comparable. Reduce $r$ in such steps that in each step $r$ is the radius of a nested interval; then as the nested interval intends to a point, the pair $(r, N(r))$ coincides itself with one-to-one correspondence. Unmeasurable set necessitates AC, but any set of reals can be measurable both if inaccessible cardinal exists and if AC is not allowed (see \cite{Solovay1970}). We prefer the latter theorem that every set is measurable, hence the comparability in $\mathfrak{R}$ is not adopted, and we obtain weak forms of the comparability as follows,\\
\indent
$\forall x, y \in \mathfrak{R}$,
\begin{enumerate}
\item $\forall x \forall y( x > y \cup x = y \cup x < y)$, \mbox{if the respective sets are measurable}.
\item $x \le y$, if $|s_x| = x$, $|s_y| = y$, and $s_x \subset s_y$.
\end{enumerate}

In general $\mathfrak{R}$ is related to $r$, while there are reals which have no connection to the pair, i.e. the constants which act on the elements. Since that $\mathcal{R}=\mathbb{R}$, hereafter let such reals denoted by $\mathbb{R}$, and let reals related to the pair denoted by $\mathcal{R}$.

\begin{definition}
Suppose that $x$, $y\in\mathfrak{R}$, define,
\begin{enumerate}
\item $x + y = (r, N_x + N_y)$, where $x = (r, N_x)$, $y = (r, N_y)$;
\item $x - y = (r, N_x - N_y)$, where $x = (r, N_x)$, $y = (r, N_y)$, and $N_x \le N_y$;
\item $x \cdot y = (r \cdot r, N_x \cdot N_y)$, if $x = (r, N_x)$, $y = (r, N_y)$;
\item $c \cdot y = (r, {N_y}^c)$, if $c \in \mathbb{R}$, $y = (r, N_y)$, and $N_y \ne 1$;
\item $c \cdot y \ne (r, 1^c)$, if $x \in \mathbb{R}$, $y = (r, 1)$.
\end{enumerate}
\end{definition}

$\mathfrak{R}$ constructs a distance space if distance in $\mathfrak{R}$ is defined as follows,
\begin{displaymath}
d(x, y) =\left \{\begin{array}{ll}
x-y & \textrm{if $x \ge y$,}\\
y-x & \textrm{if $y \ge x$,}\end{array}\right.
\end{displaymath}
since that it satisfies axioms of distance as follows:\\
\indent
$\forall x = (r, N_x(r))$, $y = (r, N_y(r))$, $z = (r, N_z(r))\in \mathfrak{R}$,
\begin{enumerate}
\item $d(x, y) = 0\iff x = y$;
\item $d(x, y) = d(y, x)$;
\item $d(x, y) \le d(x, z) + d(z, y)$.
\end{enumerate}

\subsection{Size of set using cover}
\label{sosuc}

In this subsection we define the size of set using the cover of set. The size of set is defined as follows.
\begin{definition}
Given a set $S$ and its cover $\emph{C}=\{C_i: i\in I\}$, the size of $S$, in symbols $|S|$, is $(\emph{C}, N(\emph{C}))$, where $N(\emph{C}): \emph{C} \to \mathfrak{M}$ is a set function on $\emph{C}$.
\end{definition}
\begin{remark}
\begin{enumerate}
\item $|\varnothing| = (\emph{C}, 0) = 0$;
\item $|\{0\}|=(\emph{C}, 1) > 0$;
\item For a set $U$ which is unmeasurable, $U$ is a cover for $U$ itself, then the size of $U$ is $(U, 1)$.
\end{enumerate}
\end{remark}


The cover should be invariant to translation, e.g., the cover to $\{1\}$ should be a cover to $\{2\}$ if in $\{1\}$ the element is replaced by the element in $\{2\}$. Indeed it has this property because a set is determined by the elements in the set no matter the order of the elements according to principle of extensionality on set.
\begin{definition}
\mbox{Set} $C_1$ \mbox{is a translation to set} $C_2$, \mbox{if} $[ \forall x \in C_1 \exists y \in C_2 (x + e_1 = y) \cap (\forall y \in C_2 \exists x \in C_1 (x + e_1 = y)]$, \mbox{where} $e_1 \in \mathbb{R}$ \mbox{are constants}.
\end{definition}

For an arbitrary set $S$, there exists a finite cover $\emph{C}$ on $S$.
\begin{theorem}
$\forall S \exists \emph{C}(S \subset \emph{C} \cap N(\emph{C}) < \infty)$, where $\emph{C}$ is a cover on $S$.
\end{theorem}
\begin{proof}
Recall the ring $R_0$ generated by arbitrary set $S$, i.e., $\forall S \exists !R_0 (S\subset R_0)$. Consider other rings which contain $S$, in symbols $R(S)$, which satisfy that $R_0 \subset R(S)$. Note that any subset in $R(S)$ can be covered by the union of finite subsets in $S$; denote the cover by $\emph{C}$, and $N(\emph{C})<\infty$. Hence it yields that $\forall S \exists !R_0 (S \subset R_0 \subset R(S) \subset \emph{C} (N(\emph{C}) < \infty))$.
\end{proof}

The pair $(r, N(r))$ is defined as follows.
\begin{definition}
Given a set $S$ and its size $(C, N(C))$, $|S|=(r, N(r))$ if any of $\{C_i: i \in I, C_i \in \emph{C}\}$ is a translation to each other.
\end{definition}

\begin{definition}
Given cover $\emph{C}_1$ and cover $\emph{C}_2$,  $\emph{C}_1$ is equivalent to $\emph{C}_2$ if each set in $\emph{C}_1 - \emph{C}_1 \cap \emph{C}_2$ has a translation in $\emph{C}_2 - \emph{C}_1 \cap \emph{C}_2$, and vice verso.
\end{definition}

\begin{definition}\label{eqofc}
Given set $S_1$ and set $S_2$, they are equivalent to each other in size in the sense of cover $\emph{C}=\{C_i: i\in I\}$, in symbols $|S_1|\overset{\emph{C}}{=}|S_2|$, if,
\begin{enumerate}
\item $\emph{C}$ covers $S_1$,
\item a cover equivalent to $\emph{C}$ covers $S_2$.
\end{enumerate}
\end{definition}

\begin{definition}\label{eqofsize}
$|S_1| = |S_2|$, if $|S_1|\overset{g=1}{=}|S_2|$.
\end{definition}
\begin{remark}
Definition \ref{eqofsize} is not trivial since it is designated for the precise meaning of identity, while Definition \ref{eqofc} is for the meaning of indistinguishability. Size of set varies with the scale in which the measurement is conducted. Suppose that in a scale the size of one set cannot be distinguished from the size of another set, then does this indistinguishability assure the precise identity between the two sizes? A known example is the relation between $|\{n: n\in \mathbb{N}\}|$ and $|\{2*n: n\in \mathbb{N}\}|$. They are not equivalent in the scale of $1$ while they are equivalent in the scale of $\mathbb{N}$. Is the two sets identical to each other? Not identical, but indistinguishable in the sense of $\mathbb{N}$.
\end{remark}

$\mathfrak{R}$ obeys the same rule of elementary operation on size.
\begin{definition}
Suppose that $\mathfrak{k}$, $\mathfrak{\lambda}\in \mathfrak{R}$, define,
\begin{enumerate}
\item $\mathfrak{k}+\mathfrak{\lambda}=|S_1\cup S_2|$, where $|S_1=\mathfrak{k}|$, $|S_2=\mathfrak{\lambda}|$, and $S_1\cap S_2=\varnothing$;
\item $\mathfrak{k}\cdot\mathfrak{\lambda}=|S_1\times S_2|$, where $|S_1=\mathfrak{k}|$, $|S_2=\mathfrak{\lambda}|$.\\
\end{enumerate}
\end{definition}

\section{Relation with measure theory}
\label{rwm}
\setcounter{equation}{0}

For a set $A$ and its cover $C$, suppose $\mu^*: A \to \mathfrak{R}$ is a set function on $A$, then $\mu^*$ satisfies that:
\begin{enumerate}
\item \mbox{non-negativity} $\mu^*(\varnothing) = 0$, $\mu^*(A)\ge 0$;
\item \mbox{monotonicity} \mbox{if} $A\subset B$, \mbox{then} $\mu^*(A)\le \mu^*(B)$;
\item \mbox{subadditivity} $\mu^*(\cup A_i) \le \sum\mu^*(A_i)$.
\end{enumerate}

\begin{proof}
\begin{enumerate}
\item Using any graduation except $\varnothing$ to cover $\varnothing$ produces the result as $0$, i.e., $\mu^*(\varnothing) = 0$. For another set $A$, if $A = \varnothing$, then $\mu^*(\varnothing) = 0$; if $A \ne \varnothing$, then there exists at least one element $x \in A$ to make the result bigger than $0$ at least in the sense of the graduation $1$, i.e., $\mu^*(A) \ge 0$.
\item If $A = \varnothing$, then it holds that $\mu^*(A)\le \mu^*(B)$. Since that $B = A+(B-A)$ and $A \cap (B-A) = \varnothing$, we have that $\mu^*(B)=\mu^*(A)+\mu^*(B-A)$. Using a graduation $g$ which is not $\varnothing$ but a size less than $|B|$, produces the size of $A$ as $\mu^*(A)=(g, 1)$, and produces the size of $B$ as $\mu^*(B)=(g, N(g))$ where $N(g) \ge 1$. Using a graduation $g$ larger than $|B|$, produces both $\mu^*(A)=(g, 1)$ and $\mu^*(B)=(g, 1)$. Above all, monotonicity holds.
\item Denote $\sum\mu^*(A_i)$ as $(r, N)$.Using the graduation less than $N$ produces $\mu^*(\cup A_i)\le \sum\mu^*(A_i)$, while using the graduation larger than $N$ produces $\mu^*(\cup A_i) = \sum\mu^*(A_i)$.
\end{enumerate}
\end{proof}

Then $\mu^*$ has the properties of outer measure.

In the graduation $1$, $\mu^*$ satisfies countable additivity.
\begin{equation}
\mu^*(\cup A_i) = \sum\mu^*(A_i),
\end{equation}
where $A_i \cap A_j = \varnothing$ if $i \ne j$.

\begin{proof}
If the graduation is $1$, then any element in $\cup A_i$ contributes $(r, 1)$ to $\mu^*(\cup A_i)$ in the left side of the equation, and any element in $A_i$ contributes $(r, 1)$ to $\mu^*(A_i)$ in the right side. Thus any element contributes the same size to both the left side and the right side. Since that $A_i \cap A_j = \varnothing$ if $i \ne j$, any element contributes only one time to both the left side and the right side. Then the equation holds.
\end{proof}


\section{Relation with cardinality}
\label{rwc}
\setcounter{equation}{0}

Definition \ref{eqofc} makes the definition of equivalence between cardinality as a special case. Recall that an one-to-one correspondence can be written as $a_i \leftrightarrow b_i$, i.e. $( \{a_i\}, 1)=(\{b_i\}, 1)$, where the cover to $\{a_i\}$ is a translation to the cover to $\{b_i\}$ if translation is not limited to addition. In other words, translation is a one-to-one correspondence using addition. Thus we have:
\begin{theorem}
$|S_1| = |S_2| \rightarrow \stackrel{=}{S_1} = \overset{=}{S_2}$
\end{theorem}
From the view of measurement, operations except addition violate the existence of minimum of graduation. That is, if one intends the relation between sizes precise, with the line between indistinguishability and identity drawn, he has to measure the sizes using the same graduation.

\begin{definition}\label{defofCar}
Set $S$ has cardinality $\overset{=}{S} = (C, 1)$, where $C = S$.
\end{definition}
\begin{proof}
According to definition of cardinality's equivalence, $S1 = S2$ if there exists a bijection $f: S_1\to S_2$. Bijection is defined by such $s_{1i} \leftrightarrow s_{2i}$.
\end{proof}

Recall any arbitrary set $S$ has size $(S, 1)$, and let the cardinality denoted by $\overset{=}{S}$.
\begin{lemma}
$\overset{=}{S_1}=\overset{=}{S_2}$, \mbox{if} $S_1$ \mbox{in} $(S_1, 1)$ \mbox{can be put into an one-to-one correspondence with} $S_2$ \mbox{in} $(S_2, 1)$.
\end{lemma}

In this system of size, consider the corresponding Continuum Hypothesis (CH). The first infinity in CH is $\omega$, and $\omega$ is an element in $\mathfrak{M}$ according to the definition of $\omega$. Then the set of natural number has the size $(r, \omega)$, with $r$ arbitrarily small. The set of $[0, 1]$ has the size $(r, 1/r)$, with $r$ arbitrarily small. Assume that the equity $|P(A)| = 2^{|A|}$ always holds, then the corresponding CH is:
\begin{equation}
(r, 2^\omega) = (r, 1/r) = 1.
\end{equation}
By $(r, 2^\omega) = \omega (r, 2)$, we obtain the dimension of the set $\{1, 2\}$.
\begin{equation}
\omega (r, 2) = 1.
\end{equation}
It has been believed that the first infinity is $\omega$, the second is $2^\omega$, and the third is $2^{2^\omega}$. Then the third infinity has the dimension:
\begin{equation}
(r, 2^{2^\omega}) = 2^\omega (r, 2) = 2^\omega/\omega.
\end{equation}

In summary, the corresponding CH and the equity $|P(A)| = 2^{|A|}$ yield the corresponding General Continuum Hypothesis as follows:
\begin{theorem}
\mbox{The increasing sequence of infinities has the dimensions as follows:}
\centerline{$ln \omega / ln (1/r)$, $1$, $2^\omega/\omega$, $2^{2^\omega}/\omega$, \ldots}
\end{theorem}

\vspace*{-1pt}
\section{Conclusion}
\label{conclu}
\setcounter{equation}{0}

We generalize measurement to be strong enough to measure the size of set. This measurement can vary itself with different graduations, thus it unifies the measurements with different graduations. In relation to cardinality, it constrains the one-to-one correspondence on addition, due to the existence of minimum of graduation. In relation to measure theory, it has the properties of outer measure, and even has the properties of measure if the graduation is $1$. This measurement meets the need in the generalized problem of dimension reduction, and establishes connections between dimension, cardinality and measure theory.

However, in \cite{Solovay1970} it has been shown that in Zermelo-Frankel set theory, the existence of a non-Lebesgue measurable set cannot be proved if there is an inaccessible cardinal. This suggests some connections between the size in this paper and inaccessible cardinal in cardinal theory. Moreover, in smooth infinitesimal analysis (see \cite{Hellman2006}) the infinitesimal behaves similarly to the size of a finite set.

In three aspects the size in this paper is related to mathematical philosophy or mathematical logic. First, the generalized measurement doesn't stand aside with either potential infinity or actual infinity (see \cite{Benacerraf2004}). It permits actual-infinite operations, since that an interval cannot be cut as points without the help of actual-infinite operations. For an infinite set $S$, it holds the equity of sizes between part and whole in the cases where the graduation is larger than $|S|$, a characteristic of actual infinity, while it rejects the equity in the cases where the graduation is less than $|S|$, a characteristic of potential infinity. Second, the generalized measurement distinguishes the difference between \textit{any} and \textit{all} (see \cite{Russell1908}). For a set $S$, it corresponds \textit{any} with graduation $1$, and corresponds \textit{all} with a graduation $g \le |S|$. Third, for points which are non-empty but measure zero in measure theory, how to construct the interval which has non-zero measure in measure theory? A point is measure zero in the graduation of interval, but non-empty in the graduation of $1$. This constructing problem is referred to two statements from different graduations.

\section*{Acknowledgements}
This work was supported by the National Natural Science Foundation of China under Grant 61036008.


\begin{thebibliography}{<amsplain, alpha>}\vspace*{1.8pt}
{\addtolength{\itemsep}{.8pt}}
\bibitem[Zhang, et al., 2010]{Zhang2010}{\sc Zhang, J. P., Huang, H. and Wang, J.} (2010)
Manifold learning for visualizing and analyzing high-dimensional data.
{\it IEEE Intelligent Systems}, {\bf 25}, 54--61.

\bibitem[Soltanifar, 2006]{Soltanifar2006} {\sc Soltanifar, M.} (2006)
On a sequence of Cantor fractals.
{\it Rose Hulman Undergraduate Mathematics Journal}, {\bf 7}(1), 1--9.

\bibitem[Mandelbrot, 1967]{Mandelbrot1967} {\sc Mandelrot, B. B.} (1967) How long is the coast of Britain--statistical self-similarity and fractal dimension, {\it Science}, {\bf 156}, 636--638.

\bibitem[Theiler, 1990]{Theiler1990} {\sc Theiler, J.} (1990)
Estimating fractal dimension.
{\it Journal of the Optical Society of America}, {\bf 7}(6), 1055--1073.

\bibitem[Solovay, 1970]{Solovay1970} {\sc Solovay, R. M.} (1970)
A model of set theory in which every set of reals is Lebesgue measureable.
{\it The Annals of Mathematics}, {\bf 92}(1), 1--56.

\bibitem[Benacerraf\& Putnam, 2004]{Benacerraf2004} {\sc Benacerraf, P. and Putnam, H.} (2004)
{\it Philosophy of Mathematics.} Beijing: The Commercial Press.

\bibitem[Russell, 1908]{Russell1908} {\sc Russell, B.} (1908)
Mathematical logic as based on the theory of types.
{\it American Journal of Mathematics}, {\bf 30}(3), 222--229.

\bibitem[Hellman, 2006]{Hellman2006} {\sc Hellman, G.} (2006)
Mathematical pluralism: the case of smooth infinitesimal analysis.
{\it Journal of Philosophical Logic}, {\bf 35}, 621--651.

\end{thebibliography}
\end{document}